\documentclass[12pt]{amsart}
\usepackage{enumerate}
\usepackage{hyperref}
\usepackage{color}
 \theoremstyle{plain}
 \newtheorem{theorem}{Theorem}

 \newtheorem{lemma}{Lemma}
 \newtheorem{corollary}{Corollary}
\newtheorem{proposition}{Proposition}

\theoremstyle{definition}
\newtheorem{definition}{Definition}
\newtheorem{example}{Example}
\newtheorem*{assumption}{Assumption}
 
 \newtheorem{remark}{Remark}

\newcommand{\E}{\mathbb E}

 \newcommand{\T}{\mathbb T}
 
 \newcommand{\R}{\mathbb R}

\newcommand{\cA}{\mathcal A}
 
 \newcommand{\cD}{\mathcal D}
 
  \newcommand{\HD}{H^1_\cD}

  \newcommand{\cL}{\mathcal L}

\newcommand{\cB}{\mathcal B}

 \newcommand{\tW}{\tilde{W}}

\newcommand{\af}{\alpha}
 \newcommand{\tPhi}{\tilde\Phi}
 \newcommand{\fui}{\varphi}
 \newcommand{\ep}{\varepsilon}

\newcommand{\ga}{\gamma}
\newcommand{\bga}{\bar\gamma}
 
 \newcommand{\Ga}{\Gamma}
 \newcommand{\de}{\delta}
 
 \newcommand{\om}{\omega}
 \newcommand{\Om}{\Omega}
  
\newcommand{\Lam}{\Lambda}
 
 \newcommand{\si}{\sigma}
 \newcommand{\dga}{\dot\gamma}

 \newcommand{\rip}{\rangle}
 \newcommand{\lip}{\langle}

\newcommand{\diver}{\operatorname{div}}

\newcommand{\tr}{\operatorname{Tr}}

\newcommand{\vol}{\operatorname{vol}}

\begin{document}
\title[Selection of weak KAM solution]
{A selection of a weak KAM solution of the
  sub-riemannian  Ma\~n\'e Lagrangian}
\author[I. Mart\'inez Ju\'arez]{Iker Mart\'inez Ju\'arez}
\address{Instituto de Matem\'aticas. UNAM}
\email{iker@im.unam.mx}
\author[H.  S\'anchez Morgado]{H\'ector S\'anchez Morgado}
\address{Instituto de Matem\'aticas. UNAM}
\email{hector@matem.unam.mx}
\begin{abstract}
 For  a sub-riemannian structure on the torus, satisfying the H\"ormander
condition, we consider the Ma\~n\'e Lagrangian associated to a horizontal vector field.
 Assuming that the Aubry set consists in a finite number of static classes,
 we show that the invariant measure, for the horizontal stochastic perturbation of
 the flow of  the vector field, determines a particular weak KAM
 solution of the Lagrangian, as the perturbation tends to zero.
\end{abstract}
\maketitle
\section{Preliminaries}
\label{sec:pre}
Generally, the stationary Hamilton Jacobi equation has not a unique
viscosity solution. Several approximation schemes to select a
particular viscosity solution have been studied during this century
\cite{AIPS, DFIZ, GISY}.  
This paper extends a scheme, considered in the recent paper \cite{GL}
for particular dynamics in dimension 1, to higher dimensions,
more general dynamics and to the sub-riemannian setting.

\subsection{Random perturbations}
\label{sec:random-perturbations}
For $\si_1,\ldots,\si_r$ smooth vector fields on $\T^d$, linearly
indepent at each $x\in\T^d$,  consider the bundle $\cD$ spanned by
these vector fields and set $\si=[\si_1,\ldots,\si_r]$.
We assume that $\si_1,\ldots,\si_r$ satisfy the H\"ormander condition 
\begin{equation}
\cL(\si_1,\ldots,\si_r)(x)=\R^d,~\forall x\in\T^d,\label{eq:Ho}
\end{equation}
where $\cL(\si_1,\ldots,\si_r)$
denotes the Lie algebra generated by the given vector fields.

For a smooth horizontal vector field $b$
consider the stochastic differential equation  
\begin{equation}\label{sde}
  dX^\ep_t=b(X^\ep_t)\ dt+\sqrt{\ep}\ \si(X^\ep_t)\ dW_t,\quad X^\ep_0=x
\end{equation}
where $W_t$ is an $r$-dimensional brownian motion.
For $u\in C(\T^d)$ define
\[A^\ep u(x):=\lim\limits_{t\to 0}\dfrac{\E^\ep_x[u(X^x_t)]-u(x)}t.\]
If $u\in C^2(\T^d)$ we have that
\[A^\ep u(x)=\frac{\ep}2\tr(a(x)D^2u(x))+b(x)\cdot\nabla
  u(x),\quad a(x)=\si(x)\si^t(x).\]

A generalized solution to the Cauchy problem for the backward Kolmogorov equation,
\begin{equation}\label{bKe}
  \frac{\partial u(t,x)}{\partial t}=A^\ep u(x,t),\quad u(0,x)=f(x),\quad t>0
\end{equation}
is given by
\[u(t,x)=\E^\ep_x[f(X^\ep_t)]=T_tf(x):=\int_{\T^d}P^\ep(t,x,dy)f(y)\]
where $P^\ep(t,x,\Ga)=P^\ep_x(X^\ep_t\in\Ga)$ is the law of $X^\ep_t$, which
has a density $p(t,x,y)$ that is the fundamental solution of the \eqref{bKe}.

A stationary or invariant probability $\mu_\ep$ in  $\cB(\T^d)$ is a
probability measure for which 
$\displaystyle U_t \mu_\ep:= \int_M P^\ep(t, x, dy) \mu_\ep(dx)= \mu_\ep$ for all
$t>0$. According to \cite{Kl}, under the H\"ormander condition, there
is a unique stationary probability which has a density $m_\ep(x)$ that
is a solution of the Fokker-Planck problem 
\begin{equation}\label{FP}
  \frac{\ep}2\sum_{i,j}D_{ij}^2(a^{ij}(x)m(x))-\diver(b(x)m(x))=0, \quad m(x)>0,\quad \int_{\T^d}m(x)dx=1.
\end{equation}

Letting $\fui_\ep(x)=-\dfrac\ep2\log m_\ep(x)$ , we have that $\fui_\ep(x)$ is a solution of
\begin{equation}
  \label{eq:HJv} 
 \frac{\ep}2\sum_{i,j}\frac{\ep}2 D^2_{ij}a^{ij}-[a^{ij}D_{ij}^2\fui+
  2D_ia^{ij}D_j\fui+\diver b]+  |D\fui\si|^2+bD\fui=0. 
  \end{equation}
From the theory of viscosity solutions we know that if $\fui_{\ep_n}$
converges uniformly to a function $\phi$ as $\ep_n\to 0$, then $\phi$ is a
viscosity solution of the Hamilton-Jacobi equation 
\begin{equation}
  \label{eq:HJ}
  H(x,D\phi(x))=0
\end{equation}
where $H(x,p)=|p\si(x)|^2+pb(x)$ is the Hamiltonian associated to the
sub-riemannian  Ma\~n\'e Lagrangian $L:\cD\to\R$
\begin{equation}
  \label{ML}
L(x,v)=\frac 14\|v-b(x)\|_\cD^2  
\end{equation}
where $\lip\ ,\ \rip_\cD$ is the natural metric on $\cD$  defined by 
$\lip \si(x)v,\si(x)w\rip_\cD=\lip v,w\rip$.
Since $b$ is a horizontal vector field, there is a
smooth map $\bar b:\T^d\to\R^r$ such that $b(x)=\si(x)\bar b(x)$, and so
$H(x,p)=|p\si(x)|^2+p\si(x)\bar b(x)$.

\subsection{Weak KAM theory for sub-riemannian Lagrangians.}
\label{sec:weak-kam-theory}
See \cite{SM} for details.

We recall that $f\in H^1([0,T])=H^1([0,T],\T^d)$ is called horizontal for the
bundle $\cD$ if there is $\xi\in L^2([0,T],\R^r)$ such that
$\dot f(t)=\si(f(t))\xi(t)$ a.e.
We denote by $\HD([0,T])$ the space of horizontal curves.
Under the H\"ormander condition, the metric $\lip\ ,\ \rip_\cD$ defines
a distance $d$ on $\T^d$ called {\em Carnot-Carath\'eodory distance},
which is the distance we will use in this paper.
The Carnot-Carath\'eodory  distance is topologically equivalent to the distance defined by the flat metric.

The {\em Ma\~n\'e potential} and the {\em Peierls barrier} of 
\eqref{ML} are the $d$-Lipschitz functions $\Phi, h:\T^d\times\T^d\to\R$ given by
\begin{align}\label{potential}
 \Phi(x,y)&=\inf\Big\{\int_0^T L(f,\dot f): T>0, 
f\in \HD([0,T]), f(0)=x, f(T)=y\Big\},\\
 h(x,y)&=\liminf_{T\to\infty}\min\Big\{\int_0^T L(f,\dot f):
f\in \HD([0,T]), f(0)=x, f(T)=y\Big\} \label{barrier}.
\end{align}
For each $x\in\T^d$, $h(x,\cdot)$ is a {\em weak KAM solution} for $L$.
It was shown in \cite{SM} that
\begin{itemize}
\item Weak KAM solutions for a Lagrangian, are viscosity solutions
  of \eqref{eq:HJ} for the associated Hamiltonian.  
\item If there is a continuous $F:\R\to\R^+$, such that the Hamiltoninan satisfies
\begin{equation}
  \label{eq:growth}
\Big|\frac{\partial H}{\partial x}(x,p)\Big|\le (1+|p|)F(H(x,p)),  
\end{equation}
then a viscosity solution of \eqref{eq:HJ} is a weak KAM solution.
\end{itemize}
Taking $C\ge \max |\bar b|, \max \|D\si\|, \max\|Db\|$, we have that our
Hamiltonian satisfies \eqref{eq:growth} with
$F(s)=2C(\sqrt{s_++C^2}+1)$.

The {\em Aubry set} $\cA$ of \eqref{ML} is the, non-empty,
closed set of points $x\in\T^d$ such that $h(x,x)=0$.
If $x\in\cA$ then $\Phi(x,y)=h(x,y)$ and $\Phi(y,x)=h(y,x)$
for any $y\in\T^d$.
Define the equivalence relation $\sim$ on $\cA$ by
$x\sim y\iff \Phi(x,y)+\Phi(y,x)=0$.
The equivalence classes are called {\em static classes}.

\begin{proposition} Let $\cA$ be the Aubry set of the Lagrangian \eqref{ML}.
  \begin{enumerate}
  \item The non-wandering set $\Om(b)$ of the flow of the vector field $b$
is contained in $\cA$. Therefore the $\om$-limit of any point
 is contained in $\cA$.
\item  The static classes $A_i$ are invariant under the flow of the vector field $b$.
    \end{enumerate}
\end{proposition}
\begin{proof}
  (1) If $x\in\Om(b)$ there are $f_n:[0,t_n]\to\T^d$ trajectories of $b$ such that
$f_n(0), f_n(t_n)$ converge to $x$ and $t_n\to\infty$.
By Proposition 8 in \cite{SM},
\[h(x,x)\le\lim_{n\to\infty}\displaystyle\int_0^{t_n} L(f_n,\dot f_n)=0.\]

(2) Let $x\in\cA$, there is sequence  $f_n\in\HD([0,t_n])$ with
$t_n\to\infty$ such that $f_n(0)=f_n(t_n)=x$ and
$\displaystyle\int_0^{t_n} L(f_n,\dot f_n)\to 0$. 
By Tonelli Theorem  in \cite{SM}, for each $T>0$ there is a sequence
$n^j\to\infty$ such that $f_{n^j}|[0,T]$ converges uniformly. Applying
this fact to a sequence $T_k\to\infty$ and using a diagonal
argument one gets a sequence $n_l\to\infty$ and $f\in\HD([0,\infty))$
such that $f_{n_l}|[0,T]$ converges uniformly to $f|[0,T]$ for each $T>0$ 
and $\displaystyle\int_0^T L(f,\dot f)=0$ by Theorem 2 in \cite{SM}.
Thus $f$ is a trajectory of the vector field $b$.
Since $f_{n_l}(s)\to f(s)$ and $f_{n_l}(t_{n_l})=x$, by Proposition 8 in \cite{SM},
$\displaystyle h(f(s),x)\le\lim_{l\to\infty}\int_s^{t_{n_l}}L(f_{n_l},\dot f_{n_l})=0.$
Since $x\in\cA$, $h(x,f(s))=\Phi(x,f(s))=0$, then  $f(s)\in\cA$ and  $f(s)\sim x$.
\end{proof}
We remark that the compacity of $\T^d$ plays an important role
to assure recurrence.
In this paper we will make the following
 \begin{assumption}
  The collection of static classes of the Ma\~n\'e Lagrangian \eqref{ML} is
  finite: $A_1,\ldots,A_m$. 
\end{assumption}

\begin{example}
  In $\T^3$ with local coordinates $(x_1,x_2,x_3)$ let
  \begin{align*}
    \si_1&=\cos x_3\partial_{x_1}+\sin x_3\partial_{x_2},\,\si_2=\partial_{x_3},\\
b_1&=\si_1+\sin(x_3-1)\si_2,\, b_2=\si_1+(1-\sin(x_3-1))\si_2
  \end{align*}
then $[\si_1,\si_2]=-\sin x_3\partial_{x_1}+\cos x_3\partial_{x_2}$
and so $\si_1,\si_2$ satisfy the H\"ormander condition.
The Aubry set of the Ma\~n\'e Lagrangian $\dfrac 14\|v-b_1(x)\|_\cD^2$
is the union of the static classes $A_1=\T^2\times\{(\cos 1,\sin 1)\}$
and $A_2=\T^2\times \{(-\cos 1,-\sin 1)\}$.

The Aubry set of the Ma\~n\'e Lagrangian $\dfrac 14\|v-b_2(x)\|_\cD^2$
is the static class $\T^3$.
\end{example}
If $u:\T^d\to\R$ is a viscosity solution of \eqref{eq:HJ}
then $u(p)-u(q)\le \Phi(q,p)$ for any  $p,q\in\T^d$,
so $u$ is a constant on each $A_i$, that we will denote by $u(A_i)$. Moreover
\[u(x)=\min_{1\le i\le m}u(A_i)+\Phi(A_i,x) \hbox{ for all } x.\]
Conversely, if $u_1,\ldots,u_m\in\R$ satify
$u_j-u_i\le \Phi(A_i,A_j)$ for any $i,j=1,\ldots,m$, then
\[u(x)=\min_{1\le i\le m}u_i+\Phi(A_i,x)\]
defines a viscosity solution of \eqref{eq:HJ} such that
$u(A_i)=u_i$.

\subsection{Statement of the results}\label{sec:statement-result}
As in \cite{FW}, we introduce the following concepts
that will be used in subsection \ref{sec:im} in relation with Markov chains. 
\begin{definition}
  An $\{i\}$-graph on the set $\{1\le i\le m\}$
is a graph consisting of arrows $k\to l$,
  $k,l=1,\ldots,m$, $k\ne l\ne i$ with precisely one arrow starting at
  each $k\ne i$ and without cycles. The set of $\{i\}$-graphs will be
  denoted by $G\{i\}$.
  
We say that a subset $B$ of $\T^d$ is {\em stable} if for any $x\in B$,
$y\notin B$, $\Phi(x,y)>0$. For $A_i$ stable define
\[W(A_i)=\min_{g\in G\{i\}}\sum_{(k\to l)\in g}\Phi(A_k,A_l),\]
and let $W_{\min}=\min W(A_i)$.
\end{definition}
Our general result is
\begin{theorem}\label{main}
  For $x\in\T^d$ and any $\ga>0$, there is $\rho_0>0$ such that for any
$0<\rho<\rho_0$ there is $\ep_0>0$ such that for $0<\ep<\ep_0$,
$\mu_\ep(B_\rho(x))$ is between $\exp\Big(-\dfrac{2\phi(x)\pm\ga}\ep\Big),$
where $\phi$ is the solution of \eqref{eq:HJ} given by
 \begin{equation}
    \label{eq:main}
    \phi(x)+W_{\min}  =\min_{1\le i\le m}W(A_i)+\Phi(A_i,x)
    =\min\{W(A_i)+\Phi(A_i,x):A_i\hbox{ stable}\}.
  \end{equation}
\end{theorem}
For nondegenerate diffusions we have a more precise statement
\begin{theorem}\label{main1}
  If $d=r$ and $a(x)$ is positive definite for any $x$, 
  let $\fui_\ep$ be the solution to \eqref{eq:HJv} such that
$\displaystyle\int_{\T^d}\exp(-2\fui_\ep/{\ep})=1$ and $\phi$ given by
\eqref{eq:main}. Then $\lim\limits_{\ep\to  0}\fui_\ep=\phi$ uniformly.
\end{theorem}
In section \ref{sec:prop-sol} we collect important facts. In section
\ref{sec:invariant} we estimate the invariant probability measure. In
section \ref{sec:visco} we prove that \eqref{eq:main} defines a
viscosity solution of \eqref{eq:HJ}. 
Using the estimate of the invariant probality we prove our results
in section \ref{sec:main}.

\section{Important facts}
\label{sec:prop-sol}

\subsection{The case $d=r$}
\label{sec:d=r}
\begin{proposition}\label{bernstein}
If $d=r$ and $a(x)$ is positive definite for any $x$, then the solutions
$\fui_\ep$ of \eqref{eq:HJv} have uniformly bounded derivative for $0<\ep<1$.
\end{proposition}
\begin{proof}
  Let $w_\ep=|D\fui_\ep|^2$, omitting
  the subscript $\ep$ we have
  \begin{align*}
    D_iw&=2\sum_lD^2_{ik}\fui D_k\fui,\\
    D^2_{ij}w&= 2\sum_lD^3_{ijk}\fui D_l\fui+D^2_{ik}\fui D_{jk}\fui.
      \end{align*}
  Differentiating \eqref{eq:HJv} and multiplying by $D\fui$, 
  \begin{multline*}
    \frac{\ep^2}4\sum_{i,j,k} D^3_{ijk}a^{ij}D_k\fui\\
    -\frac{\ep}2\sum_{i,j,k}[D_ka^{ij}D_{ij}^2\fui+a^{ij}D_{ijk}^3\fui+
    2D^2_{ik}a^{ij}D_j\fui+2D_ia^{ij}D_{jk}^2\fui+D_k\diver b]D_k\fui\\
    +\sum_{i,j,k}[D_ka^{ij}D_i\fui D_j\fui+2a^{ij}D^2_{ik}\fui D_j\fui+D_kb_iD_i\fui+b_iD^2_{ik}\fui]D_k\fui=0
\end{multline*}
    \begin{multline*}
    \frac{\ep^2}4\sum_{i,j,k} D^3_{ijk}a^{ij}D_k\fui-\ep\sum_{i,j,k,l}\si_l^iD_k\si^j_lD_{ij}^2\fui D_k\fui\\
    -\frac{\ep}2\sum_{i,j,k}a^{ij}(\frac 12 D^2_{ij} w-D_{ik}^2\fui D^2_{jk}\fui)+2D^2_{ik}a^{ij}D_j\fui D_k\fui+D_ia^{ij}D_jw+D_k\diver b D_k\fui\\
+\sum_{i,j,k}(D_ka^{ij}D_i\fui D_j\fui D_k\fui+a^{ij}D_iwD_j\fui+D_kb_iD_i\fui D_k\fui+
\frac12 b_iD_iw=0
\end{multline*}
Let $x_\ep\in \T^d$ be a point where $w_\ep$ attains its maximum,
then $Dw_\ep(x_\ep)=0$ and $\tr(a(x_\ep)D^2w_\ep(x_\ep))\le 0$.
  At the point $x_\ep$ we have
  \begin{multline*}
\frac{\ep}2|\si^tD^2\fui|^2\le \frac{\ep}2\sum_{i,j,k}a^{ij}D_{ik}^2\fui D^2_{kj}\fui\le\\
 -\frac{\ep^2}4\sum_{i,j,k} D^3_{ijk}a^{ij}D_k\fui+\frac{\ep}2\sum_{i,j,k,l}2\si_l^iD_k\si^j_lD_{ij}^2\fui D_k\fui+
2D^2_{ik}a^{ij}D_j\fui D_k\fui+D_k\diver b D_k\fui\\
-\sum_{i,j,k}D_ka^{ij}D_i\fui D_j\fui D_k\fui+D_kb_iD_i\fui D_k\fui,
\end{multline*}
\begin{align*}
  \frac{\ep}2|\si^tD^2\fui|^2&\le \ep^2 C|D\fui|+
\ep[\frac 14|\si^tD^2\fui|^2+C|D\fui|]+C|D\fui|^2+C|D\fui|^3,\\
\frac{\ep}2|\si^tD^2\fui\si|^2&\le \frac{2\ep}4|\si^tD^2\fui|^2|\si|^2\le \ep^2C|D\fui|+C\ep|D\fui|+C|D\fui|^2+C|D\fui|^3, \end{align*}
\begin{multline}\label{cota}
  \Big|\frac{\ep}2\sum_{i,j}\frac{\ep}2 D^2_{ij}a^{ij}(x_\ep)
  -2D_ia^{ij}(x_\ep) D_j\fui(x_\ep)-\diver b (x_\ep)+  |D\fui(x_\ep)\si(x_\ep)|^2+b (x_\ep) D\fui(x_\ep)\Big|^2\\
\le \frac{2\ep}4|\si (x_\ep)^tD^2\fui(x_\ep)|^2|\si (x_\ep)|^2\le \ep^2C|D\fui(x_\ep)|+C\ep|D\fui(x_\ep)|+C|D\fui(x_\ep)|^2+C|D\fui(x_\ep)|^3.
\end{multline}
There is $r>0$ such that
\begin{equation}\label{eq:posdef}
 |v\si(x)|^2=|va(x)v^t|\ge r\hbox{  for any }x\in\T^d, v\in\R^d, |v|=1.
\end{equation}
  Supose there is a sequence $0<\ep_n<1$ such that
$\max|D\fui_{\ep_n}|=|D\fui_{\ep_n}(x_{\ep_n})|\to\infty$, let
$v_n=D\fui_{\ep_n}(x_{\ep_n})/|D\fui_{\ep_n}(x_{\ep_n})|$. 
Dividing \eqref{cota} by $|D\fui_\ep(x_\ep)|^4$, letting $\ep=\ep_n$ and taking
limit as $n\to\infty$ we get $\lim\limits_{n\to \infty}|v_n\si(x_{\ep_n})|^2=0$,
contradicting \eqref{eq:posdef}. Therefore there is $R>0$ such that
$\max|D\fui_\ep|=|D\fui_\ep(x_\ep)|\le R$ for $0<\ep<1$.
\end{proof}
 \begin{corollary}
 If $d=r$ and $a(x)$ is positive definite for any x, then the
 solution $\fui_\ep$ of \eqref{eq:HJv} such that 
   $\displaystyle\int_{\T^d}\exp(-{2\fui_\ep}/{\ep})=1$ is uniformly
   bounded for $0<\ep<1$. 
   For any sequence $\{\ep_n\}$
   there is a subsequence of $\{\fui_{\ep_n}\}$ converging uniformly.
 \end{corollary}
\begin{proof}
  There exists $z_\ep\in \T^d$ such that $\fui_\ep(z_\ep)=0$, because if $\fui_\ep$ is always positive (negative)
  $\displaystyle\int_{\T^d}\exp(-{2\fui_\ep}/{\ep})$ is smaller(bigger) than 1.
  By Lemma  \ref{bernstein}, for $0<\ep<1$,
  $|\fui_\ep-\fui_\ep(z_\ep)|$ is uniformly bounded, and $\fui_\ep$ is equi-Lipschitz.
  By the Arzela--Ascoli theorem, for any sequence $\{\ep_n\}$
  there is a subsequence of $\{\fui_{\ep_n}\}$ converging uniformly. 
\end{proof}
\subsection{Large deviations}
\label{sec:LDP}
According to theorem 5.6.7 in \cite{DZ}, $X^\ep_t$ satisfies a large
deviation principle (LDP) in $C([0,T])=C([0,T],\T^d)$ with a good rate functional
$I_x(f)$ given by
\[I_x(f)= \frac 12\int_0^T|\dot g(t)|^2\ dt, \hbox{ if } \dot f(t)
  =b(f(t))+\si(f(t))\dot g(t), f(0)=x, g\in H^1([0,T],\R^r)\]
and $I_x(f)=+\infty$ otherwise.
This implies that denoting by $d_T$ the uniform metric in
$C([0,T])$ we have
\begin{description}
\item[LD1]For any $\de, \ga>0$ there exists $\ep_0>0$ such that
  \[    P^\ep_x(d_T(X^\ep,f)<\de)\ge\exp(-[I_x(f)+\ga]/\ep)\]
for any $0<\ep<\ep_0$, $f\in C([0,T])$, $f(0)=x$.
  \item[LD2]For $s>0$ let $F_x(s)=\{f\in C([0,T]):f(0)=x, I_x(f)\le s\}$.
For any $\de, \ga, s>0$ there exists $\ep_0>0$ such that for any $0<\ep<\ep_0$, 
\[    P^\ep_x(d_T(X^\ep, F_x(s))\ge\de)\le\exp(-[s- \ga]/\ep).\]
\end{description}
 Considering the bijection $ L^2([0,T],\R^r)\to H^1([0,T],\R^r)$
\[\xi\mapsto g(t)=\int_0^t (\xi(s) - \bar b(f(s)))\ ds\]
we have that 
\begin{align*}
  I_x(f)&=\frac 12\int_0^T|\xi(t)-\bar b(f(t))|^2\ dt\ \hbox{ if }
          \dot f(t)=\si(f(t))\xi(t), f(0)=x \xi\in L^2([0,T],\R^r),\\
          &=\frac 12 \int_0^T\|\dot f(t)-b(f(t))\|_\cD^2\ dt
            \ \hbox{ if }f\in \HD([0,T]), f(0)=x,
\end{align*}
and $I_x(f)=+\infty$ otherwise.

\subsection{Invariant measures}
\label{sec:im}
\begin{proposition}\label{sdist}
  Consider a transitive Markov chain with states $\{1,\ldots,n\}$ and transition
probabilities $P^\ep_{ij}$ then the stationary distribution of the chain is
$p_i=q_i/\sum_j q_j$ where
\[q_i=\sum_{g\in G\{i\}}\prod_{(k\to l)\in g}P^\ep_{k,l}\]
\end{proposition}
\begin{proposition}\label{invm}
  Let be given a transitive Markov chain on a phase space
  $X=\bigcup\limits_{i=1}^mX^\ep_i$, $X^\ep_i\cap X^\ep_j=\emptyset$ for $i\ne j$,
  such that there are $P^\ep_{ij}\ge 0$ $a>1$ such that transition
  probabilities of the chain satisfy
  \[a^{-1}P^\ep_{ij}\le P(x,X^\ep_j)\le aP^\ep_{ij}\quad x\in X^\ep_i, i\ne j.\]
  Then any invariant probability $\nu$ of then chain satisfies
  \[a^{2(1-m)}p_i\le \nu(X^\ep_i)\le a^{2(m-1)}p_i\]
where $p_i$ is given as in Proposition \ref{sdist}.
\end{proposition}
\section{Estimate of the invariant probability}
\label{sec:invariant}
\begin{definition}
  We extend the equivalence relation $\sim$ to $\T^d$ and call a finite
  collection $\{K_i\}$ of compact subsets of $\T^d$ {\em admissible} if
  \begin{enumerate}
  \item $x, y\in K_i\implies x\sim y$.
  \item  $x\in K_i, y\notin K_i\implies x\not\sim y$.
  \item The $\om$-limit of any point (for the flow of the vector field
    $b$) is contained in some $K_i$.
  \end{enumerate}
\end{definition}
Observe that the collection of static classes $A_1,\ldots,A_r$ is
admissible and so is the collection $A_1,\ldots,A_r,\{x\}$ for any
$x\notin \cA$.

Taking $C_b=\frac 12(1+\max\|b\|)^2$ and letting 
$f\in\HD([0,T])$ be a unitary geodesic between
$x$ and $y$, $T=d(x,y)$, we have $I_x(f)\le C_b d(x,y)$, and in particular
$\Phi(x,y)\le C_bd(x,y)/2$.
\begin{proposition}\label{Tuniform}
  For any $\ga>0$ and any compact $K$, there exists $T_0$ such that for any
  $x,y\in K_i$  there is $f\in\HD([0,T])$, $T\le T_0$, joining $x$ and $y$, such that
  $I_x(f)\le 2\Phi(x,y)+2\ga$.  
\end{proposition}
\begin{proof}
Cover $K$ by a finite number of balls $B_\de(x_i)$, $x_i\in K$, $\de<\ga/4C_b$  
Let $f_{ij}\in\HD([0,T_{ij}]$ such that $I_{x_i}(f)\le 2\Phi(x_i,x_j)+\ga$.   
Let $T_0=\max\limits_{ij}T_{ij}+2\de$. Given $x,y\in K$ choose $i,j$ such that
$x\in B_\de(x_i)$, $y\in B_\de(x_j)$, concatenating a unitary geodesic from $x$ 
to $x_i$, the curve $f_{ij}$, and unitary geodesic from $x_j$ to $y$, we obtain 
a curve $f\in\HD([0,T])$ with $T\le T_0$, joining $fx$ and $y$ such that
  \[I_x(f)\le 2 C_b\de+2\Phi(x_i,x_j)+\ga\le \frac{3\ga}2+2\Phi(x_i,x)+2\Phi(x,y)+2\Phi(y,x_i)
\le 2\Phi(x,y)+2\ga.\]
\end{proof}
Let $\{K_i\}$ be an admissible collection of compact sets and
denote $B_\de(K_i)=\{x\in \T^d:d(x,K_i)<\de\}$.

\begin{proposition}\label{small-in}
  For any $\ga, \de>0$, $x,y\in K_i$ there is $f\in\HD([0,T])$ joining
  $x$ and $y$, such that $f([0,T])\subset B_\de$ and $I_x(f)<\ga$.
\end{proposition}
\begin{proof}
  Let $f_n\in\HD([0,T])$ joining $x$ and $y$ with $I_x(f_n)$ converging to $0$. 
 If $f_n([0,T_n])\not\subset B_\de$ for all $n$, there is a subsequence
  $f_{n_k}(t_k)$ converging to $z\notin K_i$ with $x\sim z$, $y\sim z$. 
\end{proof}
\begin{lemma}\label{exit-up}
 Denote by $\tau_B$ the time of first exit of the process $X^\ep_t$ from
 the $B=B_\de(K_i)$. For any $\ga> 0$ there exists $\de> 0$ such that
 for all sufficiently small $\ep>0$, $x\in B_\de(K_i)$ we have
 \begin{equation}
   \label{eq:exit-up}
  \E_x^\ep[\tau_B]<\exp(\ga/\ep).
 \end{equation}
\end{lemma}
\begin{proof}
  Let $z\notin K_i$ such that $d(z,K_i)<\ga/3C_b$. Put $\de<d(z,K)/2$ and let
 $x\in B_\de(K_i)$. Choose $x',y\in K_i$ be such that $d(x,x_i)=d(x,K)$, $d(k,y)=d(z,K_i)$.
 By Proposition \ref{Tuniform} there is $T_0$ and $f\in\HD([0,T])$,
 with $T\le T_1$ joining $x'$ and $fy$,  such that
 $I_{x'}(f)<\ga/3$. Concatenating a unitary geodesic from $x$ to $x'$,
 the curve $f$ and a unitary geodesics from $y$ to $z$ we obtain a
 curve $f'\in\HD([0,T'])$, with $T'\le T_0=T_1+\ga/2C_b$ 
 $f'(0)=x$, $f'(T')=z$ such that $I_x(f')< 5\ga/6$.
 We extend the definition of $f'$ to $[0,T_0]$ following a 
 a trajectory of $b$. By LD1 there exists $\ep_0>0$ such that for any $0<\ep<\ep_0$
  \[P^\ep_x(\tau_B<T_0)\ge P^\ep_x(d_{T_0}(X^\ep,f)<\de)\ge\exp(-9\ga/10\ep),\]
    and then
  \[q:=\inf_{x\in B}P^\ep_x(\tau_B<T_0)\ge\exp(-9\ga/10\ep).\]
  Thus
  \begin{align*}
    P^\ep_x(\tau_B\ge (n+1)T_0)&= [1-P^\ep_x(\tau_B<(k+1)T_0|\tau_B\ge nT_0]P^\ep_x(\tau_B\ge nT_0)\\
&\le(1-q)P^\ep_x(\tau_B\ge nT_0).
  \end{align*}
and then $ P^\ep_x(\tau_B\ge (n+1)T_0)\ge (1-q)^n$.
This gives
\[\E_x^\ep[\tau_B]\le T_0\sum_{n=0}^\infty(1-q)^n=\frac{T_0}q
  \le T_0\exp(9\ga/10\ep)\le \exp(\ga/\ep)\]
if $\ep<\ep_0,\ga/10\ln T_0$.
\end{proof}
\begin{lemma}\label{exit-down}
 For any $K$ compact subset of $\T^d$ and $\ga>0$ there exists $\de> 0$ such that
 for all sufficiently small $\ep>0$, $x\in \overline B=\overline{B_\de(K)}$ we have
\begin{equation}
   \label{eq:exit-down}
  \E_x^\ep\Big[\int_0^{\tau_B}\chi_B(X^\ep_t)\ dt\Big]>\exp(-\ga/\ep).
 \end{equation}
\end{lemma}
\begin{proof}
  Let $\de<\ga/3C_b$. There is $t_0<1$ such that, any trajectory of $b$ starting at a point of $K$,
  does not leave $\overline {B_{\de/2}(K)}$ for $t\in[0,t_0]$. For $x\in\overline B$ choose $y\in K$ such that
  $T=d(x,y)=d(x,K)$ and a unitary geodesic $f\in\HD([0,T])$, between $x$ and $y$.
  Extend $f$ to $[0,T+t_0]$ by the trajectory of $b$ starting at $y$.
  By LD1 there is $\ep_0>0$ such that for any $0<\ep<\ep_0$
\[P(\tau_B>t_0)\ge P^\ep_x(d_T(X^\ep,f)\le\de/2)\ge\exp(-2\ga/3\ep).\]
Thus
\[  \E_x^\ep\Big[\int_0^{\tau_B}\chi_B(X^\ep_t)\ dt\Big]\ge t_0P(\tau_B>t_0)
\ge t_0\exp(-2\ga/\ep)>\exp(-\ga/\ep)\]
for $0<\ep<\ep_0,-\ga/3\ln t_0$.
\end{proof}

\begin{lemma}\label{enter}
Let $K\subset\T^d-\bigcup\limits_i K_i$ and $\tau_ K$ be the first exit time of $X^\ep_t$ from $K$. There are 
$c,T_0>0,\ep_0$ such that for $T>T_0$, $x\in K$, $0<\ep<\ep_0$  
\[P^\ep_x(\tau_K>T)\le \exp(-c(T-T_0)/\ep)\]
\end{lemma}
\begin{proof}
  Let $\de>0$ be such that
  $\overline B_\de(K)\subset\T^d-\bigcup\limits_i K_i$.
 For $x\in\overline B_\de(K)$ denote by $\tau(x)$ the time when the
 orbit of $b$ starting at $x$  exists  $\overline B_\de(K)$, we have $\tau(x)<\infty$.
 The function $\tau$ is upper semicontinuous, and therefore attains its maximum $T_1$ in
 $\overline B_\de(K)$.
  Let $T_0=T_1+1$ and $\cB=\{f\in\HD[0,T_0]:f([0,T_0])\subset \overline B_\de(K)\}$,
  by Tonelli theorem (see \cite{SM}),  $\displaystyle
  f\mapsto\int_0^{T_0}L(f,\dot f)$ attains its minimum $A$ in $\cB$. 
  Since there are no trajectories of $b$ in $\cB$, we have $A>0$. Let $0<\ga<A$, by LD2,
  there exists $\ep_0>0$ such that for any $0<\ep<\ep_0$, 
  \[P^\ep_x(\tau_K>T_0)\le P^\ep_x(d_T(X^\ep, F_x(A))\ge\de_0-\de_1)\le\exp(-(A-\ga)/\ep).\]
  Then
  \begin{align*}
    P^\ep_x(\tau_K> (n+1)T_0)&=\E^\ep_x[\tau_K>nT_0; P^\ep_{X^\ep_{nT_0}}(\tau_K>nT_0)]\\
    &\le \sup_{y\notin V}P^\ep_y(\tau_K>T_0) P^\ep_x(\tau_K>n T_0) 
\end{align*}
By induction
\begin{align*}
      P^\ep_x(\tau_K>nT_0)&\le\exp(-n(A-\ga)/\ep).\\
  P^\ep_x(\tau_K>T)&\le \exp\Big(-\big(\frac T{T_0}-1\big)\frac{A-\ga}\ep\Big)
\end{align*}
\end{proof}
\begin{corollary}\label{c-enter}
  For $\ep<\ep_0$, $x\in K$ we have
  \[\E_x^\ep[\tau_K]\le T_0+\ep/c\le T_0+\ep_0/c=C.\] 
\end{corollary}
For $K_i$ stable define
\begin{equation}\label{wki}
W(K_i)=\min_{g\in G\{i\}}\sum_{(k\to l)\in g}\Phi(K_k,K_l),
\end{equation}
Let $\de_0<\dfrac 12\min\limits_{i\ne j} d(K_i,K_j)$, $\de_1<\de_0$ and define
\[D=\T^d-\bigcup_iB_{\de_0}(K_i),\
  V_i=B_{\de_1}(K_i),\ V=\bigcup_iV_i  .\]
Introduce the random times $\tau_0=0$, $\si_n=\inf\{t\ge\tau_n: X^\ep_t\in D\}$,
$\tau_n=\inf\{t\ge\si_{n-1}: X^\ep_t\in\partial V\}$ and consider the
Markov chain $\{Z_n=X^\ep_{\tau_n}\}$. 
The transition probabilities and the invariant probability of the
chain $Z_n$  are  estimated using the following quantities
  \begin{align*}
    \tPhi(K_i,K_j)=&\inf\{\frac 12I_x(f):f\in\HD([0,T]), T>0,
 \forall t\ f(t)\notin\bigcup_{l\ne i,j}K_l, \\&f(0)=x\in K_i, f(T)\in K_j\}.\\
        \tW(K_i)=&\min_{g\in G\{i\}}\sum_{(k\to l)\in g}\tPhi(K_k,K_l),
                \end{align*}
 where the infimum over an empty set is taken as $+\infty$. 
Put $V_0=\max\{\tPhi(K_i,K_j)<+\infty\}$.
It is not hard to prove that
\begin{align*}
  \Phi(K_i,K_j)=&\tPhi(K_i,K_j)\wedge\min_s\tPhi(K_i,K_s)+\tPhi(K_s,K_j)\\
                &\wedge\min_{r,s}\tPhi(K_i,K_k)+\tPhi(K_k,K_s)+\tPhi(K_s,K_j)\\
&\wedge\cdots\wedge\min_{s_1,\ldots,s_{m-2}}\tPhi(K_i,K_{s_1})+\cdots+\tPhi(K_{s_{m-2}},K_j)
\end{align*}
\begin{lemma}\label{trans-est}
      For any $\ga > 0$ there exists $\de_0 > 0$ (arbitrarily small)
      such that for any $\de_2$, $0 < \de_2 < \de_0$, there exists
      $\de_1$ , $0 < \de_1 < \de_2$ such that for sufficiently small $\ep>0$,
      for all $x\in B_{\de_2}(K_i)$ the one-step transition probabilities of
      $\{Z_n\}$, $P(x,\partial V_j)$ is between
      \[\exp[-(2\tPhi(K_i,K_j)\pm\ga)/\ep].\]
    \end{lemma}
    \begin{proof}
      Let $0<\de_0<\dfrac\ga{5C_b} ,\dfrac 13\min\limits_{i\ne j} d(K_i,K_j)$.
      For $\de_2<\de_0$, fix $\de_1<\de_2, \de_0/2$.

For $i\ne j$ take $f_{ij}\in\HD([0,T_{ij}])$ such that
 $\forall t\ f(t)\notin\bigcup\limits_{l\ne i,j}K_l$, $f_{ij}(0)=z\in K_i$,
 $f_{ij}(T_{ij})\in K_j$, $I_z(f_{ij})<2\tPhi(K_i,K_j)+2\ga/5$.
 Choose $0<\de'<\de_1,\de_0-\de_2$.

We first prove the lower estimate. 
 For $x\in V_i$ take a unitary geodesic $f_1\in\HD([0,t_1])$
  joining $x$ and $x'$ with $t_1\le\de_1$, then $I_x(f_1)\le\ga/5$, $d(f_1,D)\ge\de'$.
By Proposition \ref{small-in} there is $f_2\in\HD([0,t_2])$ joining
$x'$ and $x_{ij}$, $i\ne j$ such that $I_{x'}(f_2)\le\ga/5$. Concatenating
$f_1$, $f_2$ and $f_{ij}$ we obtain a curve $g_{ij}$ joining $x$ and $K_j$
such that $I_x(g_{ij})\le 2\tPhi(K_i,K_j) +4\ga/5$. For the case $i=j$ take $g_{ij}\in\HD([0,t_0])$ a curve 
joining $x$ to $x"$ with $d(x",K_i)=\de_0+\de'$ and then to the point
$z\in K_i$ with $d(x",K_i)=\de_0+\de'$, then $I_x(g_{ij})\le 6\ga/5$.
By Propositions \ref{Tuniform}, the length of the intevals of
definition of all the curves $g_{ij}$ can be bounded by a constant $T_0$.
We extend all curves $g_{ij}$ to the interval $[0,T_0]$ following a trajectory of $b$.

If $d_{T_0}(X^\ep,g_{ij})<\de'$ for a trajectory of $X^\ep$, then it
intersects $\partial B_{\de_0}(K_i)$ and reaches $B_{\de'}(K_j)$
without intersecting $B_{\de_2+\de'}(K_l)$ for any $l\ne i,j$ and
$X^\ep_{\tau_1}\in\partial V_j$. By LD1 there is $\ep_0(\ga,T_0,\de')$
such that for $0<\ep<\ep_0$.
\[
  P(x,\partial V_j)\ge  P^\ep_x(d_T(X^\ep,f)<\de)\ge\exp(-[I_x(g_{ij})+\frac \ga5]/\ep)
  \ge\exp(-(2\tPhi(K_i,K_j)+\ga)/\ep)
\]

By the strong Markov property it is enough to prove the upper estimate
for $x\in\partial B_{\de_0}(K_i)$. By the choice of $\de_0$ and $\de'$,
for any $f\in\HD([0,t])$ with $f(0)=x\in\partial B_{\de_0}(K_i)$, with
$\ f(t)\notin\bigcup_{l\ne i,j}K_l\, \forall t\in[0,T]$,  that
reaches $B_{\de'}(\partial V_j)$ we have $2\tPhi(K_i,K_j)\le I_x(f)+3\ga/5$.
By Lemma \ref{enter} there are $T_1>0,\ep^*>0$ such that
$P^\ep_x(\tau_1>T_1)\le\exp(-2V_0/\ep)$ for $\ep<\ep^*$.
By LD2 there is $0<\ep^{**}<\ep^*$ such that
\[P^\ep_x(\tau_1\le T_1)\le P^\ep_x(d_{T_1}(X^\ep,F_x(2\tPhi(K_i,K_j)-\frac{3\ga}5))\ge\de')  
\le\exp(-(2\tPhi(K_i,K_j)-\frac{4\ga}5)/\ep)\]
for $0<\ep<\ep^{**}$. Thus

\[ P(x,\partial V_j)\le P^\ep_x(\tau_1>T_1)+P^\ep_x(\tau_1\le T_1)\le
\exp(-2V_0/\ep)+\exp (-(2\tPhi(K_i,K_j)-\frac{4\ga}5)/\ep).\]
\end{proof}
    \begin{corollary}\label{meas-chain}
      For any $\ga > 0$ there are small $0 < \de_1<\de_2 < \de_0$,
      such that for sufficiently small $\ep>0$, for the
      invariant probability $\nu_\ep$ of the Markov chain   $\{Z_n\}$,
      $\nu_\ep(\partial V_i)$ is between
      \[\exp[-2(\tW(K_i)-2W_{\min}\pm 2(m-1)\ga)/\ep]\]
    \end{corollary}
    \begin{proof}
      Applying Proposition \ref{invm} and lemma \ref{trans-est}  we
      have that $\nu_\ep(\partial V_i)$ is between
      \[\frac{q_i\exp (\mp 2(m-1)\ga/\ep)}{\sum_jq_j}, \hbox{ where }
 q_i=\sum_{g\in G\{i\}}\exp\Big[-2\sum_{(k\to l)\in g}\frac{\tPhi(K_k,K_l)}{\ep}\Big].\]
Since $\#(G\{i\})\le l(m)= \begin{pmatrix}m(m-1)\\m-1
\end{pmatrix}$, we have
\begin{align*}
  \exp[-2\tW(K_i)/\ep]&\le q_i\le l(m)\exp[-2\tW(K_i)/\ep],\\
  \exp[-2\tW_{\min}/\ep]&\le\sum_jq_j\le ml(m)
 \exp[-\tW_{\min}/\ep].
\end{align*}
  Thus
    \begin{multline*}
      \frac{\exp\Big[-2(\tW(K_i)-W_{\min}+ (m-1))\ga/\ep]}{ml(m)}
     \le \nu_\ep(K_i)\\\le l(m)\exp\Big[-2(\tW(K_i)-W_{\min}-(m-1))/\ep,] 
   \end{multline*}
   and taking $\ep>0$ sufficiently small we get the assertion of the corollary.
 \end{proof}
  \begin{proposition} $W(K_i)=\tW(K_i)$  
  \end{proposition}
  \begin{proof}
   The inequality $W(K_i)\le\tW(K_i)$ is clear. Let $g\in G\{i\}$ give the minimum
in \eqref{wki} and let $k\to l$ be in $g$. If $\Phi(K_k,K_l)\ne\tPhi(K_k,K_l)$, we have
\[\tPhi(K_k,K_l)=\tPhi(K_k,K_{s_1})+\cdots+\tPhi(K_{s_n},K_l)\]
and replace $k\to l$ by $k\to s_1\to\cdots\to s_n\to l$.
The sum does not change but the new graph is not an
$\{i\}$-graph. The index $i$ may be one the $s_j$ and if that is the
case we omit the arrow starting at $i$. The new graph may have a cycle
$l\to\ s_j\to\cdots\to s_n\to l$, so in $g$ there was a different
arrow $s_k\to r, k>j$. We omit the new arrow $s_k\to s_{k+1}$ or
$s_n\to l$. There are possibly points $s_j$ with two arrows starting
at this point, and  we now omit the old arrows. Having performed the
omissions the sum does not increase. Repeating this procedure we
arrive at a graph $\tilde g$ for which
\[\sum_{(k\to l)\in\tilde  g}\tPhi(K_k,K_l)\le \sum_{(k\to l)\in g}\Phi(K_k,K_l)\]
\end{proof}

 \begin{theorem}\label{FW}
  Let $\mu_\ep$ be the invariant probability
  measure of the diffusion  process $(X^\ep_t,P^\ep_x)$ defined by
  \eqref{sde}. Then for any  $\ga> 0$ there exists $\de> 0$ (which
  can be chosen arbitrarily small) such that for $x\in K_i$ 
  we have that  $\mu_\ep(B_\de(x))$ is between
  \[\exp[-2(W (K_i) - W_{\min} \pm\ga)/\ep].\]
\end{theorem}
\begin{proof}
  According to theorem 4.1 in \cite{Kh}, $\mu_\ep$ is given up to a factor by
  \begin{equation}\label{eq:kh}
\mu_\ep(A)=\int_{\partial V}\nu_\ep(dy)\ \E_y^\ep\Big[\int_0^{\tau_1}\chi_A(X^\ep_t)\ dt\Big]
  \end{equation}
where $\nu_\ep$ is the invariant probability of the Markov chain $\{Z_n\}$.

Applying lemmas \ref{exit-up}, \ref{exit-down} and corollary
\ref{meas-chain} for $\ga/4m$ we have that $\mu_\ep(B_\de(x))$ is between 
$\exp \Big[-\Big (2W(K_i)-2W_{\min}\pm\dfrac{2m-1}{2m}\ga \Big)/\ep \Big]$. When
$W_{\min}=W(K_i)$ we have 
\[\mu_\ep(\T^d)\ge \mu_\ep(B_\de(K_i))\ge \exp\Big[-\frac{(2m-1)\ga}{2m\ep}\Big].\]

From \eqref{eq:kh} we have
\begin{align*}
  \mu_\ep(\T^d)&=\int_{\partial V}\nu_\ep(dy)\ \E_y^\ep[\tau_1]=
              \int_{\partial V}\nu_\ep(dy)( \E_y^\ep[\si_0]+\E_y^\ep[\E_{X^\ep_{\si_0}}^\ep\tau_1])\\
&\le \sup_{y\in\partial V}\E_y^\ep[\si_0]+\sup_{x\in D}\E_x^\ep[\tau_1]\le
\exp[\ga/2m\ep]+C
\end{align*}
The first term on the r.h.s. comes from \eqref{eq:exit-up} and the
second one from Corollary \ref{c-enter}. 
Normalizing $\mu_\ep$  dividing by $\mu_\ep (\T^d)$, we obtain the assertion of the theorem.
\end{proof}
  \section{A viscosity solution with values $W(A_i)$}
\label{sec:visco}
  \begin{proposition}
  For $K_j$ unstable there exists $K_i$ stable such that $\Phi(K_j,K_i)=0$.
\end{proposition}
\begin{proof}
  There is $x\notin K_j$ such that $\Phi(K_j,x)=0$. The $\om$-limit
  of $x$ is contained in some $K_l$ and $\Phi(K_j,K_l)=0$. Since
  $x\notin K_i$, $l\ne j$. If $K_l$ is unstable we repeat the argument
  to obtain $K_{J(0)}=K_j,\ldots,K_{J(r)}$ all different with $\Phi(K_{J(l)},K_{J(l+1)})=0$.
Thus, eventually  we arrive to a $K_i$ stable.   
\end{proof}
\begin{proposition}\label{properties-w}
  \begin{enumerate}[(a)]
  \item  There is $g\in G\{i\}$ for which the minimum in
    \eqref{wki} is attained and for  each unstable
    $K_k$, the arrow $k\to l$ in $g$ corresponds to
    $K_l$ stable with $\Phi (K_k,K_l) = 0$.
  \item For a stable $K_i$, the value $W(K_i)$ can be calculated 
    using graphs only on the set of indices of stable $K_j$.
  \item $W(K_j) = \min\{W(K_i)+\Phi(K_i,K_j):K_i\hbox{ stable}\}$. 
  \end{enumerate}
\end{proposition}
\begin{proof}
  (a) Take a graph in $G\{i\}$ the minimum in \eqref{wki} is attained
  and suppose $K_k$ is unstable and the arrow $k\to l$ does not
  satisfy the condition.
  
  If there are no arrows $r\to k$ with $K_r$ unstable we replace $k\to
  l$ by $k\to j$ with $K_j$ stable such that $\Phi(K_k,K_j)=0$.

  If there are not any arrow ending at $k$ we get a graph $g\in
 G\{i\}$ such that $\displaystyle\sum_{(r\to n)\in g}\Phi(K_r,K_n)$ decreases.

If the arrows there arrows $s_1\to k,\ldots, s_n\to k$ with $K_1,\ldots, K_n$ stable
and no cycles are formed we get a graph in $G\{i\}$ such that the
corresponding sum decreases. If a cycle $k\to j\cdots,\to s_r\to k$ is
formed we replace $s_r\to k$ by $s_r\to l$. We  get a graph in $G\{i\}$ and 
since
\[\Phi(K_k,K_j)+\Phi(K_{s_r},K_l)=\Phi(K_{s_r},K_l)\le
  \Phi(K_{s_r},K_k) +\Phi(K_k,K_l),\]
the corresponding sum decreases.

Repeating this procedure we get rid  of the wrong arrows.

(b) Es clear that the minimum calculated 
    using graphs only on the set of indices of stable $K_j$ is not
    greater than $W(K_i)$. For the opposite inequality, for any graph on
    the set of indices of stable $K_j$, we add for each unstable
    $K_j$ an arrow $j\to r$ with $K_r$ stable such that $\Phi(K_j,K_i)=0$,
to obtain a graph in $G\{i\}$ keeping the same sum.

(c) Let $i\ne j$ and take a graph in $G\{i\}$ where the minimum in
  \eqref{wki} is  attained, which has the arrow $j\to n$.
  Omitting $j\to n$ we get two
graphs, one ending at $j$ and the other ending at $i$ (possibly one of
them reduces to a point). The sum of $\Phi(K_k,K_l)$ over
all remaining arrows $k\to l$ is $W(K_i)-\Phi(K_j,K_l)$. Adding now the arrow
$i\to j$ we get a graph $g\in G\{j\}$ such that
\begin{align}\nonumber
  W(K_j)&\le\sum_{(k\to l)\in g}\Phi(K_k,K_l)=W(K_i)-\Phi(K_j,K_n)+\Phi(K_i,K_j)\\
   &\le W(K_i)+\Phi(K_i,K_j).\label{dominated}
\end{align}

  Thus $W(K_j) = \min\limits_{1\le i\le m}W(K_i)+\Phi(K_i,K_j)
\le\min\{W(K_i)+\Phi(K_i,K_j):K_i\hbox{ stable}\}$. 

If $K_j$ is stable, $W(K_j) = \min\{W(K_i)+\Phi(K_i,K_j):K_i\hbox{ stable}\}$.

If $K_j$ is unstable choose $K_r$ stable such that $\Phi(K_j,K_r)=0$ and a graph $g\in G\{i\}$ for which 
$W(K_i)=\displaystyle\sum_{(k\to l)\in g}\Phi(K_k,K_l)$ and for any $K_k$ unstable the arrow $k\to l$ corresponds to
$K_l$ stable with $\Phi(K_k,K_l)=0$.
Adding the arrow $j\to r$, a cycle $j\to r\to\cdots\to s\to j$ with $K_s$ stable
is formed. Omitting the arrow $s\to j$ we obtain a graph $h\in G\{s\}$ such that
\begin{align*}
W(K_j)=\sum_{(k\to l)\in h}\Phi(K_k,K_l)+\Phi(K_s,K_j)&\ge W(K_s)+\Phi(K_s,K_j)\\
  &\ge\min\{W(K_i)+\Phi(K_i,K_j):K_i\hbox{ stable}\}.
\end{align*}
\end{proof}

Take the admissible collection $\{K_i\}$ to be the set of static
classes $\{A_i\}$.
\begin{corollary}
  There is a viscosity solution $\psi$ of \eqref{eq:HJ} with
  $\psi(A_i)=W(A_i)$, $i=1,\ldots, m$ and given by
  \begin{equation}\label{eq:sol}
      \psi(x)=\min_{1\le i\le m}W(A_i)+\Phi(A_i,x)
      =\min\{W(A_i)+\Phi(A_i,x):A_i\hbox{ stable}\}
      \end{equation}
\end{corollary}
\begin{proof}
  The existence of a viscosity solution $\psi$ of \eqref{eq:HJ} with
  values $\psi(A_i)=W(A_i)$ and given by the first equality in
  \eqref{eq:sol}, follows from \eqref{dominated}.
  Let $A_j$ be unstable, $x\in \T^d$ and choose $A_i$ stable such that
  $W(A_j)=W(A_i)+\Phi(A_i,A_j)$, then 
 \[W(A_j)+\Phi(A_j,x)=W(A_i)+\Phi(A_i,A_j)+\Phi(A_j,x)\ge W(A_i)+\Phi(A_i,x)\]
which proves the second equality in \eqref{eq:sol}.
\end{proof}
If $x\notin\cA$, taking the admissible collection
$A_1,\ldots,A_m,\{x\}$, we have that $\{x\}$ is unstable, and therefore 
$W(\{x\})$ is  $\psi(x)$ given by \eqref{eq:sol}.

\section{Proof of the results}
\label{sec:main}
\subsection{Proof of theorem \ref{main}}
\label{sec:proof1}
For $x\in\cA$ take the admissible collection of static classes
$\{A_i\}$ and apply Theorem \ref{FW}. 
 For $x\notin\cA$ take the admissible collection $A_1,\ldots,A_m,\{x\}$,
and apply again Theorem \ref{FW}. 
\subsection{Proof of theorem \ref{main1}}
\label{sec:proof2}
Let  $\ep_n>0$ be a sequence converging to $0$, there is a subsequence
$\{\fui_{\ep_{n_k}}\}$ that converges uniformly to a viscosity solution
$\phi$ of \eqref{eq:HJ}.
Let $\ga>0$ and choose $\de>0$ according to Theorem \ref{FW}
and such that $|\fui_{\ep_{n_k}}(x)-\phi(A_i)|<\ga$ for $k$ large and $x\in V_i$.
For $k$ large $\mu_{\ep_{n_k}}(V_i)$ is between
\[\vol(V_i)\exp[-2(\phi(A_i)\pm\ga)/\ep_{n_k}]\]
and also between
 \[\exp[-2(W (A_i) - W_{\min} \pm\ga)/\ep_{n_k}].\]
 Therefore
 \[|\phi(A_i) -W(A_i)+W_{\min}|\le 2\ga-\ep_{n_k}\log \vol(V_i)/2,\]
 and so $|\phi(A_i)-W(A_i)+W_{\min}|\le 2\ga$ for any $\ga>0$. Thus 
 $\phi(A_i)=W(A_i)-W_{\min}$ and then $\phi(x)=\psi(x)-W_{\min}$ where
$\psi$ is given by \eqref{eq:sol}.
 

\begin{thebibliography}{99}
   \bibitem[AIPS]{AIPS}
\newblock N. Anantharaman, R. Iturriaga, P. Padilla and H. S\'anchez Morgado.
\newblock   Physical solutions of the Hamilton-Jacobi equation.
  \newblock {\em Disc. Cont. Dyn. Sys. Series B}, {\bf 5}
(2005), 513--528.

\bibitem[DZ]{DZ}
  \newblock A. Dembo \& O. Zeitouni.
  \newblock {\em Large Deviations Techniques and Applications.}
  \newblock Stochastic Modelling and Applied Probability 38. Springer.
\bibitem[FW]{FW}
   \newblock M. I. Freidlin \& A. D. Wentzell.
   \newblock{\em Random Perturbations of Dynamical Systems.}
   \newblock GMW 260. Springer.

   
 \bibitem[DFIZ]{DFIZ}
   \newblock A. Davini, A. Fathi, R. Iturriaga \& M. Zavidovique.
\newblock Convergence of the solutions of the discounted Hamilton–Jacobi equation.
\newblock {\em Invent. math}. {\bf 206} (2016) 29–-55.

 \bibitem[GL]{GL}
   \newblock Y. Gao \& J. Liu. A selection principle for weak KAM
   solutions via Freidlin-Wentzell large deviation principle of
   invariant measures.
   \newblock {\em SIAM J. Math. Anal.} {\bf 55}, No. 6,
   (2023) 6457--6495.
   
 \bibitem[GISY]{GISY}
   \newblock D. Gomes, Iturriaga R., S\'anchez Morgado H., Y. Yu.
   \newblock Mather measures selected by an approximation scheme.
   \newblock {\em Proc. Amer. Math. Soc.} {\bf 138} (2010), no. 10, 3591–3601.
 \bibitem[Kh]{Kh}
   \newblock R. Z. Khasminskii.
   \newblock {\em Stochastic Stability of Differential Equations}.
   \newblock Stochastic Modelling and Applied Probability 66. Springer
   
 \bibitem[Kl]{Kl} \newblock W. Kliemann.
 \newblock Recurrence and Invariant Measures for Degenerate Diffusions.
 \newblock {\em Annals of Probability} {\bf 15}, No. 2, (2024) 690--707.
 
\bibitem[SM]{SM} \newblock H. S\'anchez Morgado.
  \newblock Homogenization for sub-riemannian Lagrangians.
  \newblock {\em Nonlinearity}  {\bf 36} (2023) 3043--3067.
\end{thebibliography}
\end{document}